\documentclass[runningheads]{llncs}

\usepackage[T1]{fontenc}
\def\doi#1{\href{https://doi.org/\detokenize{#1}}{\url{https://doi.org/\detokenize{#1}}}}
\usepackage{graphicx}
\usepackage[breaklinks,colorlinks,citecolor=blue,urlcolor=blue]{hyperref}
\usepackage{listings}
\newcommand{\1}{\vspace{.1in}}

\newcommand{\bc}{\begin{center}}
	\newcommand{\ec}{\end{center}}
\newcommand{\n}{\noindent}

\usepackage{amsmath}
\usepackage{amssymb}
\usepackage{bussproofs}
\usepackage[T1]{fontenc}
\usepackage{url}

\begin{document}
	
\title{Towards Proof-Theoretic Formulation of the General Theory of Term-Forming Operators\thanks{Funded by the European Union (ERC, ExtenDD, project number: 101054714). Views and opinions expressed are however those of the author(s) only and do not necessarily reflect those of the European Union or the European Research Council. Neither the European Union nor the granting authority can be held responsible for them.
}}

\titlerunning{General Theory of Term-Forming Operators}

\author{Andrzej Indrzejczak\inst{1}\orcidID{0000-0003-4063-1651}}
	
\authorrunning{Andrzej Indrzejczak}


%

\institute{Department of Logic, University of Lodz, Poland
	\email{andrzej.indrzejczak@filhist.uni.lodz.pl}}

\maketitle              

	\begin{abstract}
		Term-forming operators (tfos), like iota- or epsilon-operator, are technical devices applied to build complex terms in formal languages. Although they are very useful in practice their theory is not well developed. In the paper we provide a proof-theoretic formulation of the general approach to tfos provided independently by several authors like Scott, Hatcher, Corcoran, and compare it with an approach proposed later by Tennant. Eventually it is shown how the general theory can be applied to specific areas like Quine's set theory NF.

		
	\end{abstract}
	
	
	\section{Introduction}
	
	In formal languages terms are usually treated as these elements of language which only refer to the objects in the domain of discourse. In particular, this way of treating terms is prevailing in proof theory and automated deduction where usually only functional terms are approved.
	In contrast, in natural languages, naming expressions are used very often not only for referring to objects but also for conveying information about them. In the earlier stages of development of mathematical logic several formal devices were introduced for this aim which currently are rather neglected. These term-forming operators,
also called shortly tfos or vbtos  (variable binding term operators), include, among others:

\begin{itemize}
\item 
iota-operator (Peano): $\imath x\varphi$ - the (only) $x$ such that $\varphi$;
\item
epsilon-operator (Hilbert): $\epsilon x\varphi$  - a(n) $x$ such that $\varphi$;
\item
abstraction-operator: $\{x:\varphi\}$ - the set of (all) $x$ satisfying $\varphi$;
\item
counting-operator (Frege): $\sharp x\varphi$ - the number of $x$ such that $\varphi$;
\item
lambda-operator (Church): $\lambda x\varphi$ - the property of being $\varphi$.
\end{itemize}

It seems that currently only the lambda-operator is treated as an important tool and found diverse applications in recursion theory, type theory and proof theory. Abstraction-operator, although commonly used in practice, is rather not treated seriously in the formal development of set theories. The remaining ones are sadly treated as formal tools having only some historical value. Since the role of complex terms as information conveying tools is crucial in communication it is important to fill this gap.

Recently, some more attention was paid to proof theory of definite descriptions. In particular, cut-free sequent calculi were provided for Fregean \cite{Indrzejczak2019}, Russellian \cite{Indrzejczak2021c} and free description theories \cite{Indrzejczak2020b}. The latter theories were also characterised in terms of tableau systems \cite{IndZaw1} and tableau calculus was also used to develop a Russelian theory in the language enriched with lambda-operator \cite{IndZaw2}. Some modal logics of definite descriptions were also developed in terms of cut-free sequent calculus \cite{Indrzejczak2018a}, in particular, the logic of Fitting and Mendelsohn \cite{FitMen98} was independently formalised as a labelled sequent calculus \cite{Orlandelli2021} and as a hybrid system \cite{Indrzejczak2020a}. Alternatively, interesting natural deduction and sequent calculi were proposed for free and intuitionistic logics of definite descriptions characterised in terms of binary quantifier \cite{Kurbis2019a,Kurbis2019b,Kurbis2021a,Kurbis2021b,Kurbis2021c}.

Since definite descriptions are amenable to proof theoretic treatment it is tempting to suspect that for other tfos we can obtain equally interesting results. Perhaps one should start with posing a question whether a general theory of such operators is possible? In fact at least two different attempts to develop such a theory were proposed.
The earlier approach was independently introduced by several authors, including: Scott \cite{scott}, Da Costa \cite{DaCosta73,DaCosta80}, Hatcher \cite{Hatcher1968,Hatcher1982}, Corcoran and Herring \cite{Corcoran71,Corcoran72}. It was formulated semantically and as an axiomatic theory. In what follows it will be called simply S-theory (after Scott). 
The second approach was introduced by Neil Tennant \cite{Tennant1978}, and then developed in \cite{Tennant2004} as a general theory of abstraction operators (see also \cite{Tennant1987,Tennant2022}). This T-theory was formulated in terms of natural deduction system and with adequate semantical characterisation. In what follows we will examine these two approaches and show how they can be formulated as well-behaved sequent calculi in section 3. Then, in section 4 we consider their specification with respect to set-abstraction operator. For this aim we focus on Quine's version of set theory NF (New Foundations) \cite{Quine} (see also \cite{Rosser}) but the proposed systems may be modified to apply to other formulations of set theory as well.

	\section{Preliminaries}

	We will be using standard first-order predicate languages with quantifiers $\forall, \exists$, identity predicate $=$ and arbitrary term-forming operator $\tau$ making complex terms from formulae of the language. The definition of a term and formula is standard, by simultaneous recursion on both categories. In the presented system the only terms are variables and complex terms constructed by means of arbitrary unary tfo $\tau$. The complex terms are written as $\tau x\varphi$ where $\varphi$ is a formula in the scope of respective operator.
	
	In accordance with Gentzen's custom we divide individual variables into bound $VAR = \{x,y,z, \ldots \}$ and free variables (parameters) $PAR = \{a, b, c, \ldots \}$. It makes easier an elaboration of some technical issues concerning substitution and proof transformations. In the metalanguage $\varphi, \psi, \chi$ denote any formulae and $\Gamma, \Delta, \Pi, \Sigma$ their multisets. Metavariables $t, t_1, \ldots$ denote arbitrary terms. $\varphi[t_1/t_2]$ is officially used for the operation of correct substitution of a term $t_2$ for all occurrences of a term $t_1$ (a variable or parameter) in $\varphi$, and similarly $\Gamma[t_1/t_2]$ for a  uniform substitution in all formulae in $\Gamma$. Ocassionally, we will use simplified notation $\varphi(t)$ to denote the result of correct substitution.
	
	First-order logic in general will be abbreviated as FOL or FOLI if identity is primitive. CFOL(I), PFFOL(I), NFFOL(I) denote the classical, positive free and negative free versions. 
	The basic system GC for CFOL consists of the following rules:
	
	\1
	
		\n\begin{tabular}{lll}
			
			$(Cut)$ \ $\dfrac{\Gamma \Rightarrow \Delta, \varphi \qquad \varphi , \Pi \Rightarrow \Sigma  }{\Gamma, \Pi \Rightarrow \Delta, \Sigma }$ &
			$(AX)$ \ $ \varphi, \Gamma  \Rightarrow \Delta, \varphi$ & \\[16pt]

			$(\neg\!\!\Rightarrow)$ \ $\dfrac{\Gamma \Rightarrow \Delta, \varphi}{\neg\varphi,
				\Gamma \Rightarrow \Delta}$ &
			$(\Rightarrow\!\!\neg)$ \ $\dfrac{\varphi, \Gamma
				\Rightarrow \Delta}{\Gamma
				\Rightarrow \Delta, \neg\varphi}$ & $(W\!\!\Rightarrow)$ \ $\dfrac{\Gamma \Rightarrow \Delta}{\varphi, \Gamma \Rightarrow \Delta}$
			\\[16pt]

			$(\Rightarrow\!\!\wedge)$ \ $\dfrac{\Gamma
				\Rightarrow \Delta, \varphi \qquad
				\Gamma \Rightarrow \Delta,
				\psi}{\Gamma \Rightarrow \Delta,
				\varphi\wedge\psi}$ & $(\wedge\!\!\Rightarrow)$ \ $\dfrac{\varphi, \psi,
				\Gamma \Rightarrow \Delta}{\varphi\wedge\psi, \Gamma
				\Rightarrow \Delta}$ & 	$(\Rightarrow\!\!W)$ \ $\dfrac{\Gamma \Rightarrow \Delta}{\Gamma\Rightarrow \Delta, \varphi}$
			\\[16pt]

			$(\vee\!\!\Rightarrow)$ \ $\dfrac{\varphi, \Gamma
				\Rightarrow \Delta \qquad \psi, \Gamma
				\Rightarrow \Delta}{\varphi\vee\psi,
				\Gamma\Rightarrow \Delta}$ &
			$(\Rightarrow\!\!\vee)$ \ $\dfrac{\Gamma
				\Rightarrow \Delta, \varphi,\psi} {\Gamma\Rightarrow \Delta,
				\varphi\vee\psi}$ & $(C\!\!\Rightarrow)$ \ $\dfrac{\varphi , \varphi , \Gamma \Rightarrow \Delta}{\varphi, \Gamma \Rightarrow \Delta}$
			\\[16pt]
			
			$(\rightarrow \Rightarrow)$ \ $\dfrac{\Gamma
				\Rightarrow \Delta, \varphi \qquad\psi,
				\Gamma \Rightarrow \Delta} {\varphi\rightarrow\psi, \Gamma
				\Rightarrow \Delta}$ &
			$(\Rightarrow \rightarrow)$ \ $\dfrac{\varphi, \Gamma
				\Rightarrow \Delta, \psi}{\Gamma\Rightarrow\Delta, \varphi\rightarrow\psi}$ & 	$(\Rightarrow\!\!C)$ \ $\dfrac{\Gamma \Rightarrow \Delta, \varphi , \varphi}{\Gamma \Rightarrow \Delta, \varphi}$ \\[16pt]

			$(\leftrightarrow\Rightarrow)$ \ $\dfrac{\Gamma \!\!\Rightarrow
				\Delta, \varphi, \psi \qquad \varphi, \psi, \Gamma \!\!\Rightarrow \Delta
			}{\varphi\!\leftrightarrow\!\psi, \Gamma \!\!\Rightarrow
				\Delta }$ & $(\forall \!\!\Rightarrow)$ \ $\dfrac{\varphi[x/t], \Gamma \!\!\Rightarrow \Delta
			}{\forall x\varphi, \Gamma \!\!\Rightarrow
				\Delta }$ & $(\Rightarrow\!\!\exists
			)$ \ $\dfrac{\Gamma \!\!\Rightarrow
				\Delta, \varphi[x/t]
			}{\Gamma \!\!\Rightarrow
				\Delta, \exists x\varphi}$ \\[16pt]

			$(\Rightarrow\leftrightarrow)$ \
			$\dfrac{\varphi, \Gamma \!\!\Rightarrow \Delta,
				\psi \qquad \psi, \Gamma \Rightarrow \Delta, \varphi}{\Gamma \!\!\Rightarrow \Delta, 
				\varphi\!\leftrightarrow\!\psi}$  &
			$(\Rightarrow\!\!\forall)$ \ $\dfrac{\Gamma \!\!\Rightarrow
				\Delta, \varphi[x/a]}{\Gamma \!\!\Rightarrow \Delta,
				\forall x\varphi}$ & 
			$(\exists\!\!\Rightarrow)$ \ $\dfrac{\varphi[x/a], \Gamma
				\!\!\Rightarrow \Delta}{\exists x\varphi, \Gamma
				\!\!\Rightarrow \Delta }$ \\[16pt]

		\end{tabular}

	\n	where $a$ is a fresh parameter (eigenvariable), not present in $\Gamma, \Delta$ and $\varphi$.

	If instead of  $(\forall \!\!\Rightarrow)$ and $(\Rightarrow\!\!\exists
	)$ we introduce:
	
	\medskip

	\n\begin{tabular}{ll}
		
		$(\forall \!\!\Rightarrow)$ \ $\dfrac{\varphi[x/b], \Gamma \!\!\Rightarrow \Delta
		}{\forall x\varphi, \Gamma \!\!\Rightarrow
			\Delta }$ & $(\Rightarrow\!\!\exists
		)$ \ $\dfrac{\Gamma \!\!\Rightarrow
			\Delta, \varphi[x/b]
		}{\Gamma \!\!\Rightarrow
			\Delta, \exists x\varphi}$ \\[16pt]
		
	\end{tabular}
	
	\n we obtain a pure variant GPC which is adequate for CFOL with variables as the only terms but in general incomplete for extensions with some tfos.

	The variant GF for PFFOL can be obtained by changing all quantifier rules into:
	
	\1
	
	\n\begin{tabular}{ll}
		
		$(\forall \!\!\Rightarrow)^F$ \ $\dfrac{\varphi[x/t], \Gamma \!\!\Rightarrow \Delta
		}{Et, \forall x\varphi, \Gamma \!\!\Rightarrow
			\Delta }$  &
		$(\Rightarrow\!\!\forall)^F$ \ $\dfrac{Ea, \Gamma \!\!\Rightarrow
			\Delta, \varphi[x/a]}{\Gamma \!\!\Rightarrow \Delta,
			\forall x\varphi}$\\[16pt]
		
		$(\exists\!\!\Rightarrow)^F$ \ $\dfrac{Ea, \varphi[x/a], \Gamma
			\!\!\Rightarrow \Delta}{\exists x\varphi, \Gamma
			\!\!\Rightarrow \Delta }$ & $(\Rightarrow\!\!\exists
		)^F$ \ $\dfrac{\Gamma \!\!\Rightarrow
			\Delta, \varphi[x/t]
		}{Et, \Gamma \!\!\Rightarrow
			\Delta, \exists x\varphi}$ \\[16pt]

	\end{tabular}
	
	\n where $E$ is the existence predicate, which is usually defined as $Et := \exists x(x=t)$. This form of rules follows from the fact that in free logics terms may designate nonexistent objects whereas quantifiers have existential import. For pure version GPF again we use $b$ instead of $t$ in $(\forall \!\!\Rightarrow)^F$ and $(\Rightarrow\!\!\exists)^F$.
	
	Moreover, in negative free logic atomic formulae with such terms are false which implies that 
	$Et \rightarrow t=t$
	and 	
	$\varphi(t)\rightarrow Et$, 
	for any atomic formula $\varphi$.
	Hence to obtain GNF (or GPNF) for NFFOL we have to add to GF (or GPF) the rule requiring all predicates to be strict in the sense that they are satisfied only by denoting terms:
	
	\1
	
	$(Str)$ \ $\dfrac{Et, \Gamma \!\!\Rightarrow \Delta
	}{\varphi(t), \Gamma \!\!\Rightarrow
		\Delta }$ \hspace{1cm} where $\varphi$ is atomic.
	
	\1
	
	Identity can be characterised in GC (GPC) and GF (GPF) in several ways (see \cite{Ind:Synt}). For our purposes we use the following rules:
	
	\1
	
	\n\begin{tabular}{ll}
	
	$(Ref)$ \ $\dfrac{t=t, \Gamma \!\!\Rightarrow
			\Delta}{\Gamma \!\!\Rightarrow\Delta}$ & \ \ $(2LL)$ \ $\dfrac{\Gamma
		\Rightarrow \Delta, t_1=t_2 \qquad
	\Gamma \Rightarrow \Delta, \varphi[x/t_1]}{\Gamma \Rightarrow \Delta,
\varphi[x/t_2]}$\\[16pt]

\end{tabular}
	
\n where $\varphi$ is atomic.

GCI, GPCI, GFI, GPFI will denote the respective calculi with the rules for identity added.
In case of NFFOLI, due to strictness condition, reflexivity does not hold unconditionally and we must weaken the first rule, using instead:

\1

	$(Ref)^N$ \ $\frac{\mbox{$t=t, \Gamma \!\!\Rightarrow
		\Delta$}}{\mbox{$Et, \Gamma \!\!\Rightarrow
		\Delta$}}$

\1

GNFI, GPNFI will denote the respective calculi for NFFOLI with the rules for identity having 
$(Ref)^N$.

Proofs are defined in the standard way as finite trees with nodes labelled by sequents. The height of a proof ${\cal D}$ of $\Gamma\Rightarrow\Delta$ is defined as the number of nodes of the longest branch in ${\cal D}$. $\vdash_k \Gamma\Rightarrow\Delta$ means that $\Gamma\Rightarrow\Delta$ has a proof with height at most $k$.
Let us recall that formulae displayed in the schemata are active, whereas the remaining ones are parametric, or form a context. In particular, all active formulae in the premisses are called side formulae, and the one in the conclusion is the principal formula of the respective rule application. 

Note that the Cut-elimination theorem holds for all above mentioned calculi (see e.g. \cite{Indrzejczak2021b}) and the full Leibniz' Law LL: $t_1=t_2, \varphi[x/t_1]\Rightarrow \varphi[x/t_2]$ (for arbitrary formula $\varphi$) is also provable.

	\section{The General Theory}

The S-theory of tfos is expressed by two general principles:

\medskip

EXT: $\forall x(\varphi(x)\leftrightarrow\psi(x))\rightarrow \tau x\varphi(x)=\tau x\psi(x)$

AV: $\tau x\varphi(x)=\tau y\varphi(y)$

\medskip

\n or, equivalently, by one principle:

\medskip

EXTAV: $\forall xy(x=y\rightarrow(\varphi(x)\leftrightarrow\psi(y)))\rightarrow \tau x\varphi(x)=\tau y\psi(y)$

\1

Such a general theory was first developed on the basis of positive free first-order logic  with identity by Scott \cite{scott}. However, the remaining authors used the classical first-order logic with identity as the basis. In both cases the general completeness theorem was provided and several important model theoretic results which hold for CFOLI (see in particular Da Costa \cite{DaCosta80}). In what follows, we will pay more attention to classical case since for several kinds of tfos, in particular for descriptions, it is rather difficult to find reasonable theories, in contrast to the situation in free logic (see \cite{Lambert}). 

Several possible objections can be raised against such a theory.
In a sense it is too general and too weak, on the other hand, for specific kind of operators it may be too strong, in particular in the setting of classical logic. Let us illustrate these remarks with some examples. For example, for $\imath$-operator Rosser \cite{Rosser} is enforced to add (in CFOLI) to EXT and AV the following axiom:

\1

$\exists_1x\varphi(x)\rightarrow\forall x(x=\iota x\varphi(x)\leftrightarrow\varphi(x))$

\1

\n which still gives incomplete logic as noticed by Hailperin \cite{Hailperin}.
Da Costa \cite{DaCosta80} adds:

\1

$\exists_1x\varphi(x)\rightarrow\forall x(x=\iota x\varphi(x)\rightarrow\varphi(x))$
\ \ and

$\neg\exists_1x\varphi(x)\rightarrow\iota x\varphi(x)= \iota x(x\neq x)$

\medskip

In fact, the theory of descriptions axiomatised by the addition of these two axioms to EXT and AV is redundant, since the latter principles can be proven with their help. This theory is in fact equivalent to Fregean/Carnapian theory of descriptions (often called the chosen object theory), in particular in the formulation of Kalish and Montague \cite{KaliMonMar80}. However, we call an S-theory every theory of arbitrary tfo where EXT and AV hold either as axioms or as derived theses.

On the other hand, for some theories of definite descriptions these two principles are too strong. For example, in the Russellian theory \cite{Russell1905,WhitRus} both principles do not hold. Instead we have their weaker versions:

\medskip

wEXT: $E\imath x\varphi(x) \rightarrow ((\varphi(x)\leftrightarrow\psi(x))\rightarrow \imath x\varphi(x)=\imath x\psi(x))$

\medskip

wAV: $E\imath x\varphi(x)\rightarrow\imath x\varphi(x) = \imath y\varphi(y)$.

\medskip

In other cases of tfos, like set-abstraction operator or counting operator, EXT may be even more disastrous, since for the latter it yields one half of the Fregean ill-famed V law, in fact this half which is sufficient for deriving contradiction. Similar problems with set-abstraction will be discussed below.

\subsection{The Formalisation of S-theory}

To obtain an adequate sequent calculus for S-theory we add to GCI the following two rules:

\1

\n\begin{tabular}{ll}

$(Ext)$ \ $\dfrac{\varphi(a), \Gamma
	\Rightarrow \Delta, \psi(a) \qquad
	\psi(a), \Gamma \Rightarrow \Delta, \varphi(a)}{\Gamma \Rightarrow \Delta,
	\tau x\varphi(x) = \tau x\psi(x)}$ &  $(AV)$ \ $\dfrac{\tau x\varphi(x) = \tau y\varphi(y),  \Gamma \!\!\Rightarrow
		\Delta}{\Gamma \!\!\Rightarrow
		\Delta}$ \\[16pt]

\end{tabular}

\n where $a$ is a fresh parameter.

Alternatively, we can add just one rule corresponding to EXTAV:

\begin{prooftree}
	\AxiomC{$a=b, \varphi(a), \Gamma \Rightarrow \Delta, \psi(b)$}
	\AxiomC{$a=b, \psi(b), \Gamma\Rightarrow \Delta, \varphi(a)$}
	\LeftLabel{$(ExtAV)$}
	\BinaryInfC{$\Gamma\Rightarrow\Delta, \tau x\varphi(x) = \tau y\psi(y)$}
\end{prooftree}

\n where both $a, b$ are fresh parameters.
 
\begin{theorem}
	GCI+$\{(Ext), (AV)\}$ and GCI+$\{(ExtAV)\}$ are equivalent to axiomatic formulations of S-theory of tfos.
\end{theorem}

\begin{proof}
It is sufficient to prove respective axioms in GCI+$\{(Ext), (AV)\}$ or in GCI+$\{(ExtAV)\}$ and to show that the above rules are derivable in GCI with added axioms EXT, AV or EXTAV. We will show this for the more compact version with $(ExtAV)$ and EXTAV; proofs for the remaining rules and axioms are similar and simpler. 
Provability of EXTAV:

\begin{prooftree}
	\AxiomC{$a=b \Rightarrow a=b$}
	\AxiomC{$\varphi(a)\leftrightarrow\psi(b), \varphi(a)\Rightarrow \psi(b)$}
	\LeftLabel{$(\rightarrow\Rightarrow)$}
	\BinaryInfC{$a=b\rightarrow(\varphi(a)\leftrightarrow\psi(b)), a=b, \varphi(a)\Rightarrow \psi(b)$}
	\LeftLabel{$(\forall\Rightarrow)$}
	\UnaryInfC{$\forall xy(x=y\rightarrow(\varphi(x)\leftrightarrow\psi(y))), a=b,  \varphi(a)\Rightarrow \psi(b)$}
	\AxiomC{${\cal D}$}
	\LeftLabel{$(ExtAV)$}
	\BinaryInfC{$\forall xy(x=y\rightarrow(\varphi(x)\leftrightarrow\psi(y)))\Rightarrow \tau x\varphi(x)=\tau y\psi(y)$}
\end{prooftree}

\n where the rightmost leaf is provable and ${\cal D}$ is an analogous proof of $\forall xy(x=y\rightarrow(\varphi(x)\leftrightarrow\psi(y))), a=b,  \psi(b)\Rightarrow \varphi(a)$.

Derivability of $(ExtAV)$:

\begin{prooftree}
	\AxiomC{$a=b, \varphi(a), \Gamma\Rightarrow\Delta, \psi(b)$}
	\AxiomC{$a=b, \psi(b), \Gamma\Rightarrow\Delta, \varphi(a)$}
	\LeftLabel{$(\Rightarrow\leftrightarrow)$}
	\BinaryInfC{$a=b, \Gamma\Rightarrow\Delta, \varphi(a)\leftrightarrow\psi(b)$}
	\LeftLabel{$(\Rightarrow\rightarrow)$}
	\UnaryInfC{$\Gamma\Rightarrow\Delta, a=b\rightarrow (\varphi(a)\leftrightarrow\psi(b))$}
	\LeftLabel{$(\Rightarrow\forall)$}
	\UnaryInfC{$\Gamma\Rightarrow\Delta, \forall xy(x=y\rightarrow(\varphi(x)\leftrightarrow\psi(y)))$}
	\AxiomC{${\cal D}$}
	\LeftLabel{$(Cut)$}
	\BinaryInfC{$\Gamma\Rightarrow\Delta, \tau x\varphi(x)=\tau y\psi(y)$}
\end{prooftree}

\n where both leaves are provable instances of LL and ${\cal D}$ is a proof of $\forall xy(x=y\rightarrow(\varphi(x)\leftrightarrow\psi(y)))\Rightarrow\tau x\varphi(x)=\tau y\psi(y)$ from the axiom $\Rightarrow EXTAV$. \qed \end{proof}

\1

Let us consider the question of cut elimination for either of the two formalisations of S-theory. 
We can observe that the choice of the rule $(2LL)$ for representation of LL was connected with the shape of $(Ext)$ or $(ExtAV)$. In both calculi identities can appear as the principal formulae of some rule application only in the succedent. This makes it safe for proving cut elimination since identities in antecedents can only appear either as parametric formulae or as formulae introduced by weakening. In both cases if identity is a cut formula under consideration it is eliminable either by induction on the height of cut or directly.

Still there is a problem connected with the application of $(\forall\Rightarrow)$ and $(\Rightarrow\exists)$ to complex terms. If for example $\forall x\varphi$ is a cut formula which was in both premisses of cut introduced as the principal formula, and in the right premiss $x$ was instantiated with $\tau y\psi$, then the formula $\varphi[x/\tau y\psi]$ may have higher complexity than $\forall x\varphi$ and the induction on the complexity of cut formulae fails. This problem  may be overcome either by introduction of more complex way of measuring the complexity of formulae (see e.g. \cite{Indrzejczak2019}) or by replacing the basic calculus GCI with its pure version GPCI. Of course, the restriction of all quantifier rules to parameters makes the calculus with complex terms incomplete. However, to avoid the loss of generality we can add to GPCI the rule:

\begin{prooftree}
	\AxiomC{$a = \tau x\varphi(x), \Gamma \Rightarrow \Delta$}
	\LeftLabel{$(a\Rightarrow)$}
	\UnaryInfC{$\Gamma\Rightarrow\Delta$}
\end{prooftree}

\n where $a$ is a fresh parameter.

\begin{theorem}
	The calculus GPCI+$\{(Ext), (AV)\}$ (or GPCI+$\{(ExtAV)\}$) with added $(a\Rightarrow)$ is equivalent to GCI+$\{(Ext), (AV)\}$ (or GCI+$\{(ExtAV)\}$)
\end{theorem}

\begin{proof}
It is enough to show that $(a\Rightarrow)$ is derivable in GCI:

\begin{prooftree}
	\AxiomC{$\Rightarrow \tau x\varphi(x) = \tau x\varphi(x)$}
	\LeftLabel{$(\Rightarrow\exists)$}
	\UnaryInfC{$\Rightarrow \exists y(y = \tau x\varphi(x))$}
	\AxiomC{$a = \tau x\varphi(x), \Gamma \Rightarrow \Delta$}
	\RightLabel{$(\exists\Rightarrow)$}
	\UnaryInfC{$\exists y(y=\tau x\varphi(x)), \Gamma\Rightarrow\Delta$}
	\LeftLabel{$(Cut)$}
	\BinaryInfC{$\Gamma\Rightarrow\Delta$}
\end{prooftree}

\n and that unrestricted $(\forall\Rightarrow), (\Rightarrow\exists)$ are derivable in GPC with $(a\Rightarrow)$:

\begin{prooftree}
	\AxiomC{$\Gamma\Rightarrow\Delta, \varphi(\tau x\psi(x))$}
	\AxiomC{$\varphi(\tau x\psi(x)), a=\tau x\psi(x) \Rightarrow \varphi(a)$}
	\LeftLabel{$(Cut)$}
	\BinaryInfC{$a = \tau x\psi(x), \Gamma\Rightarrow\Delta, \varphi(a)$}
	\LeftLabel{$(\Rightarrow\exists)$}
	\UnaryInfC{$a = \tau x\psi(x), \Gamma\Rightarrow \Delta, \exists x\varphi$}
	\LeftLabel{$(a\Rightarrow)$}
	\UnaryInfC{$\Gamma\Rightarrow\Delta, \exists x\varphi$}
\end{prooftree}

\n where the rightmost sequent being an instance of LL is provable.
Similar proof works for $(\forall\Rightarrow)$. \qed \end{proof}

\1

Let us call GPCI+$\{(Ext), (AV)\}$ (or GPCI+$\{(ExtAV)\}$) with added $(a\Rightarrow)$ simply GS (GS'). Note that for both systems the following lemma holds:

\begin{lemma}\label{leibnizandsubstitution}
	\begin{enumerate}
		\item
		$\vdash t_1=t_2, \varphi[x/t_1] \Rightarrow\varphi[x/t_2]$, for any formula $\varphi$.
		\item
		If $\vdash_k \Gamma \Rightarrow \Delta$, then $\vdash_k \Gamma[b_1/b_2] \Rightarrow \Delta[b_1/b_2]$, where $k$ is the height of a proof.  
	\end{enumerate}	
\end{lemma}

\noindent \begin{proof} 1. follows by induction on the complexity of $\varphi$
	and is standard for all cases.
	The proof of 2 is by induction on the height of proofs. \qed \end{proof}

\medskip

The first result is Leibniz' Law (LL) stated in full generality, i.e. covering also complex terms. Since $(2LL)$ yields only LL restricted to atomic formulae, we need its unrestricted form for completeness. The second result is a substitution lemma which is necessary for unifying terms while proving the cut elimination theorem. Note that it is restricted to parameters only but in the case of GS (GS'), which is an extension of GPCI, it is sufficient since only parameters are instantiated for bound variables in all applications of quantifier rules.

\begin{theorem} 
	The cut elimination theorem holds for GS and GS'.
\end{theorem}

\noindent \begin{proof}
The proof is standard and essentially requires two inductions: on the complexity of cut formula and on the height of the derivations of both premisses of cut. In general we can follow the strategy applied for example in \cite{Indrzejczak2021b}; here we focus only on the crucial points connected with the new rules which could lead to troubles. 

Consider the situation where the cut formula in the left premiss is the principal formula of the application of $(2LL)$.  It is an atomic formula, possibly an identity. Since in no logical rule atomic formula in the antecedent can be a principal formula, so in the right premiss a suitable cut formula is either introduced by weakening or is just a parametric formula. In the first case it is directly eliminated, in the second it is eliminated by induction on the height of the proof. The case where the right premiss is axiomatic is also directly eliminable.

The cases where in the left premiss the cut formula is the principal formula of the application of $(Ext)$ or $(ExtAV)$ are treated in a similar way. Eventually, rules like $(AV)$ or $(a\Rightarrow)$ have no impact on the elimination of cuts since there are no principal formulae in the conclusion.
\qed \end{proof}

Although we cannot totally avoid the loss of the subformula property in GS and GS', the introduction of complex terms is separated from quantifier rules and technically it is more desirable. In fact, from the semantic point of view we are not really in need of introducing an arbitrary complex term in the premiss while doing a proof-search. The rule is required only for these terms which either occur already in $\Gamma, \Delta$, or have in their scope the formulae from $\Gamma, \Delta$. It can be shown by providing Hintikka-style completeness proof for this system which is possible since Henkin-style proofs were provided by the mentioned authors; we omit the details because of space restrictions. 

In fact, for the needs of proof-search we could simplify GS (GS') a little bit. In particular we could use a more convenient one-premiss rule of Negri and von Plato \cite{neg:str01} for LL of the form:

\begin{prooftree}
	
	\AxiomC{$\varphi(t_2), \Gamma\Rightarrow \Delta$}
	\LeftLabel{$(1LL)$}
	\UnaryInfC{$t_1=t_2, \varphi(t_1), \Gamma\Rightarrow\Delta$}
\end{prooftree}

\n for all cases where at least one of $t_1, t_2$ is a parameter and $\varphi(t_1)$ is not an identity with both arguments being complex terms. In fact, the only troublesome cases of LL which could make a clash in the proof of cut elimination are three:

\begin{enumerate}
	\item $b=t, t=t' \Rightarrow b=t'$
	\item $t=t', \varphi(t) \Rightarrow \varphi(t')$
	\item $t=t', t'=t'' \Rightarrow t=t''$
\end{enumerate}

\n where $t, t'$ are complex terms, and only for these cases a two-premiss rule $(2LL)$ is necessary.

Also note that instead of $(Ref)$ we can use more restricted version:

\begin{prooftree}
	
	\AxiomC{$b=b, \Gamma\Rightarrow \Delta$}
	\LeftLabel{$(Ref')$}
	\UnaryInfC{$\Gamma\Rightarrow\Delta$}
\end{prooftree}

\n since $\tau x\varphi(x)=\tau x\varphi(x)$ is derivable by $(Ext)$ or $(ExtAV)$.

\subsection{The Formalisation of T-theory}

The theory of abstraction-operators developed by Tennant, which we call here a T-theory of tfos, is generally much stronger than S-theory. But we must emphasize that it is formulated in the setting of much weaker logic, namely NFFOLI (negative free FOLI), where not only quantifier rules are weaker but also the identity is not (unconditionally) reflexive. 

Tennant's theory of tfo is based on the following natural deduction rules:

\1

$(\tau I)$ \ If $\varphi(a), Ea\vdash aRt$ and $aRt\vdash\varphi(a)$ and $Et$, then $t=\tau x\varphi(x)$; 

$(\tau E1)$ If $t=\tau x\varphi(x)$ and $\varphi(b)$ and $Eb$, then $bRt$

$(\tau E2)$ If $t=\tau x\varphi(x)$, then $Et$

$(\tau E3)$ If $t=\tau x\varphi(x)$ and $bRt$, then $\varphi(b)$

\1

\n where $a$ is an eigenvariable, and $R$ is a specific relation involved in the characterisation of $\tau$. For example, $R$ is $=$ for the case of $\imath$, and $\in$ for set-abstraction operator. 
The corresponding sequent rules are:

\begin{prooftree}
	\AxiomC{$\Gamma \Rightarrow \Delta, Et$}
	\AxiomC{$Ea, \varphi(a), \Gamma \Rightarrow \Delta, aRt$}
	\AxiomC{$aRt, \Gamma \Rightarrow \Delta, \varphi(a)$}
	\LeftLabel{$(\Rightarrow\tau)$}
	\TrinaryInfC{$\Gamma \Rightarrow \Delta, t=\tau x\varphi(x)$}
\end{prooftree}

\n where $a$ is not in $\Gamma, \Delta, \varphi$

\begin{prooftree}
	\AxiomC{$\Gamma\Rightarrow \Delta, Eb$}
	\AxiomC{$\Gamma\Rightarrow \Delta, \varphi(b)$}
	\AxiomC{$\Gamma\Rightarrow \Delta, t=\tau x\varphi(x)$}
	\LeftLabel{$(\Rightarrow\tau E1)$}
	\TrinaryInfC{$\Gamma \Rightarrow \Delta, bRt $}
\end{prooftree}

\begin{prooftree}
	\AxiomC{$\Gamma \Rightarrow \Delta, t=\tau x\varphi(x)$}
	\LeftLabel{$(\Rightarrow\tau E2)$}
	\UnaryInfC{$\Gamma \Rightarrow \Delta, Et $}
\end{prooftree}

\begin{prooftree}
	\AxiomC{$\Gamma \Rightarrow \Delta, bRt$}
	\AxiomC{$\Gamma \Rightarrow \Delta, t=\tau x\varphi(x)$}
	\LeftLabel{$(\Rightarrow\tau E3)$}
	\BinaryInfC{$\Gamma \Rightarrow \Delta, \varphi(b)$}
\end{prooftree}

To get more standard SC we can apply the rule-generation theorem (see e.g. \cite{ind:sat})
and obtain left introduction rules for $\tau$:

\begin{prooftree}
	\AxiomC{$\Gamma \Rightarrow \Delta, Eb$}
	\AxiomC{$\Gamma \Rightarrow \Delta, \varphi(b)$}
	\AxiomC{$bRt, \Gamma \Rightarrow \Delta$}
	\LeftLabel{$(\tau \Rightarrow 1)$}
	\TrinaryInfC{$t=\tau x\varphi(x), \Gamma \Rightarrow \Delta $}
\end{prooftree}

\begin{prooftree}
	\AxiomC{$ Et, \Gamma \Rightarrow \Delta$}
	\LeftLabel{$(\tau \Rightarrow 2)$}
	\UnaryInfC{$t=\tau x\varphi(x), \Gamma \Rightarrow \Delta $}
\end{prooftree}

\begin{prooftree}
	\AxiomC{$\Gamma \Rightarrow \Delta, bRt$}
	\AxiomC{$\varphi(b), \Gamma \Rightarrow \Delta$}
	\LeftLabel{$(\tau\Rightarrow 3)$}
	\BinaryInfC{$t=\tau x\varphi(x), \Gamma \Rightarrow \Delta$}
\end{prooftree}

Note that if we transfer these rules to the setting of CFOLI we do not need formulae of the form $Et$, and the rule $(\tau\Rightarrow 2)$, being specific to negative free logic, is superfluous.
As a result we obtain the following three rules:

\begin{prooftree}
	\AxiomC{$\varphi(a), \Gamma \Rightarrow \Delta, aRt$}
	\AxiomC{$aRt, \Gamma \Rightarrow \Delta, \varphi(a)$}
	\LeftLabel{$(\Rightarrow\tau)$}
	\BinaryInfC{$\Gamma \Rightarrow \Delta, t=\tau x\varphi(x)$}
\end{prooftree}

\n where $a$ is not in $\Gamma, \Delta, \varphi$

\begin{prooftree}
	\AxiomC{$\Gamma \Rightarrow \Delta, \varphi(b)$}
	\AxiomC{$bRt, \Gamma \Rightarrow \Delta$}
	\LeftLabel{$(\tau \Rightarrow)$}
	\BinaryInfC{$t=\tau x\varphi(x), \Gamma \Rightarrow \Delta $}
\end{prooftree}

\begin{prooftree}
	\AxiomC{$\Gamma \Rightarrow \Delta, bRt$}
	\AxiomC{$\varphi(b), \Gamma \Rightarrow \Delta$}
	\LeftLabel{$(\tau\Rightarrow)$}
	\BinaryInfC{$t=\tau x\varphi(x), \Gamma \Rightarrow \Delta$}
\end{prooftree}

In general what we obtain with these rules is equivalent to the following principle, often called Lambert axiom:

\1

LA: \ $\forall y(y = \tau x\varphi(x) \leftrightarrow \forall x(\varphi(x)\leftrightarrow xRy))$

\1

\n which is derivable also in the setting of NFFOLI. In the setting of CFOLI it is equivalent to Hintikka axiom: 

\1

HA: \ $t = \tau x\varphi(x) \leftrightarrow \forall x(\varphi(x)\leftrightarrow xRt)$

\1

\n for which we demonstrate syntactically the equivalence with the stated rules. 
In one direction we have:

\begin{scriptsize}
	\begin{prooftree}
		\AxiomC{$\varphi[x/a] \Rightarrow \varphi[x/a]$}
		\AxiomC{$aRt \Rightarrow aRt$}
		\LeftLabel{$(\tau\Rightarrow)$}
		\BinaryInfC{$t=\tau x\varphi(x), \varphi[x/a] \Rightarrow aRt $}
		\AxiomC{$aRt\Rightarrow aRt$}
		\AxiomC{$ \varphi[x/a] \Rightarrow \varphi[x/a]$}
		\BinaryInfC{$t=\tau x\varphi(x), aRt \Rightarrow \varphi[x/a] $}
		\LeftLabel{$(\Rightarrow\leftrightarrow)$}
		\BinaryInfC{$t=\tau x\varphi(x) \Rightarrow \varphi[x/a]\leftrightarrow aRt $}
		\LeftLabel{$(\Rightarrow\forall)$}
		\UnaryInfC{$t=\tau x\varphi(x) \Rightarrow \forall x(\varphi(x)\leftrightarrow xRt) $}
	\end{prooftree}
\end{scriptsize}

In the second direction: 

\begin{scriptsize}
	\begin{prooftree}
		\AxiomC{$aRt \Rightarrow aRt$}
		\AxiomC{$\varphi[x/a]\Rightarrow \varphi[x/a]$}
		\LeftLabel{$(\leftrightarrow\Rightarrow)$}
		\BinaryInfC{$\varphi[x/a]\leftrightarrow aRt, aRt \Rightarrow\varphi[x/a] $}
		\LeftLabel{$(\forall\Rightarrow)$}
		\UnaryInfC{$\forall x(\varphi(x)\leftrightarrow xRt), aRt \Rightarrow\varphi[x/a] $}
		\AxiomC{$\varphi[x/a] \Rightarrow \varphi[x/a]$}
		\AxiomC{$aRt\Rightarrow aRt$}
		\BinaryInfC{$\varphi[x/a]\leftrightarrow aRt, \varphi[x/a]\Rightarrow aRt $}
		\UnaryInfC{$\forall x(\varphi(x)\leftrightarrow xRt), \varphi[x/a]\Rightarrow aRt $}
		\LeftLabel{$(\Rightarrow\tau)$}
		\BinaryInfC{$\forall x(\varphi(x)\leftrightarrow xRt)\Rightarrow t=\tau x\varphi(x) $}
	\end{prooftree}
\end{scriptsize}

Derivability of the specific rules is straightforward. Notice that from HA as additional axioms we obtain:

\1

(a) $t=\tau x\varphi(x)\Rightarrow \forall x(\varphi(x)\leftrightarrow xRt)$  \ \ \ \ and

(b) $\forall x(\varphi(x)\leftrightarrow xRt)\Rightarrow t=\tau x\varphi(x) $.

\1

From the premisses of any variant of $(\tau\Rightarrow)$, applying weakening we deduce:

\begin{prooftree}
	\AxiomC{$\Gamma \Rightarrow \Delta, bRt, \varphi[x/b]$}
	\AxiomC{$bRt, \varphi[x/b], \Gamma \Rightarrow \Delta$}
	\LeftLabel{$(\leftrightarrow\Rightarrow)$}
	\BinaryInfC{$\varphi[x/b]\leftrightarrow bRt, \Gamma \Rightarrow\Delta $}
	\LeftLabel{$(\forall\Rightarrow)$}
	\UnaryInfC{$\forall x(\varphi(x)\leftrightarrow xRt), \Gamma \Rightarrow\Delta $}
\end{prooftree}

\n which, by cut with (a) yields the conclusion of $(\tau\Rightarrow)$. In a similar way we deduce $\Gamma\Rightarrow \Delta, \forall x(\varphi(x)\leftrightarrow xRt)$ from premisses of $(\Rightarrow\tau)$, and by cut with (b) we obtain the conclusion of this rule.

One should note that T-theory is much stronger than S-theory; both central principles EXT and AV are provable (in fact even in the setting of NFFOLI by means of the weaker rules). 

\begin{scriptsize}
	\begin{prooftree}
		\AxiomC{$aR\tau x\varphi(x) \Rightarrow aR\tau x\varphi(x)$}
		\AxiomC{$\varphi[x/a], \varphi[x/a]\leftrightarrow \psi[x/a] \Rightarrow \psi[x/a]$}
		\LeftLabel{$(\tau\Rightarrow)$}
		\BinaryInfC{$\tau x\varphi(x)=\tau x\varphi(x), \varphi[x/a]\leftrightarrow \psi[x/a], aR\tau x\varphi(x) \Rightarrow\psi[x/a] $}
		\LeftLabel{$(Ref)$}
		\UnaryInfC{$\varphi[x/a]\leftrightarrow \psi[x/a], aR\tau x\varphi(x) \Rightarrow\psi[x/a] $}
		\LeftLabel{$(\forall\Rightarrow)$}
		\UnaryInfC{$\forall x(\varphi(x)\leftrightarrow \psi(x)), aR\tau x\varphi(x) \Rightarrow\psi[x/a] $}
		\AxiomC{${\cal D}$}
		\LeftLabel{$(\Rightarrow\tau)$}
		\BinaryInfC{$\forall x(\varphi(x)\leftrightarrow \psi(x))\Rightarrow \tau x\varphi(x)=\tau x\psi(x) $}
	\end{prooftree}
\end{scriptsize}

\n where the second leaf is directly provable and ${\cal D}$ is an analogous proof of $\forall x(\varphi(x)\leftrightarrow \psi(x)), \psi[x/a]\Rightarrow aR\tau x\varphi(x) $.

\begin{scriptsize}
	\begin{prooftree}
		\AxiomC{$aR\tau x\varphi(x) \Rightarrow aR\tau x\varphi(x)$}
		\AxiomC{$\varphi[x/a] \Rightarrow \varphi[y/a]$}
		\LeftLabel{$(\tau\Rightarrow)$}
		\BinaryInfC{$\tau x\varphi(x)=\tau x\varphi(x),  aR\tau x\varphi(x) \Rightarrow\varphi[y/a] $}
		\LeftLabel{$(Ref)$}
		\UnaryInfC{$aR\tau x\varphi(x) \Rightarrow\varphi[y/a] $}
		\AxiomC{$\varphi[y/a] \Rightarrow \varphi[x/a]$}
		\AxiomC{$aR\tau x\varphi(x) \Rightarrow aR\tau x\varphi(x)$}
		\BinaryInfC{$\tau x\varphi(x)=\tau x\varphi(x),  \varphi[y/a] \Rightarrow aR\tau x\varphi(x) $}
		\UnaryInfC{$\varphi[y/a]\Rightarrow aR\tau x\varphi(x) $}
		\LeftLabel{$(\Rightarrow\tau)$}
		\BinaryInfC{$\Rightarrow \tau x\varphi(x)=\tau y\varphi(y) $}
	\end{prooftree}
\end{scriptsize}

Note that $\varphi[x/a]$ and $\varphi[y/a]$ are identical since $\varphi(x)$ and $\varphi(y)$ are alphabetic variants.

One may even prove the converse of EXT:

\begin{scriptsize}
	\begin{prooftree}
		\AxiomC{$\varphi[x/a] \Rightarrow \varphi[x/a]$}
		\AxiomC{$aR\tau x\varphi(x) \Rightarrow aR\tau x\varphi(x)$}
		\LeftLabel{$(\tau\Rightarrow)$}
		\BinaryInfC{$\tau x\varphi(x)=\tau x\varphi(x),  \varphi[x/a] \Rightarrow aR\tau x\varphi(x) $}
		\LeftLabel{$(Ref)$}
		\UnaryInfC{$\varphi[x/a] \Rightarrow aR\tau x\varphi(x)$}
		\AxiomC{$\psi[x/a] \Rightarrow \psi[x/a]$}
		\LeftLabel{$(\tau\Rightarrow)$}
		\BinaryInfC{$\tau x\varphi(x)=\tau x\psi(x), \varphi[x/a] \Rightarrow \psi[x/a]$}
		\AxiomC{${\cal D}$}
		\BinaryInfC{$\tau x\varphi(x)=\tau x\psi(x)  \Rightarrow \varphi[x/a]\leftrightarrow \psi[x/a] $}
		\LeftLabel{$(\Rightarrow\forall)$}
		\UnaryInfC{$\tau x\varphi(x)=\tau x\psi(x) \Rightarrow \forall x(\varphi(x)\leftrightarrow \psi(x)) $}
	\end{prooftree}
\end{scriptsize}

\n where ${\cal D}$ is a similar proof of $\tau x\varphi(x)=\tau x\psi(x), \psi[x/a]  \Rightarrow \varphi[x/a] $.

To realise how strong is this principle on the ground of CFOLI notice that when $t$ is instantiated with $\tau x\varphi(x)$ we obtain:

\1

$\tau x\varphi(x) = \tau x\varphi(x) \leftrightarrow \forall x(\varphi(x)\leftrightarrow xR\tau x\varphi(x))$.

\1

which by (unrestricted) reflexivity of = yields:

\1

$\forall x(\varphi(x)\leftrightarrow xR\tau x\varphi(x))$.

\1

For several term-forming operators, at least on the ground of CFOLI, it is too strong. For example if we instantiate this principle with iota-operator (where $R$ is = ) we run into contradiction:

\1

\begin{footnotesize}
	
1. $\imath x(Ax\wedge\neg Ax) = \imath x(Ax\wedge\neg Ax) \rightarrow \forall x(Ax\wedge\neg Ax\leftrightarrow x = \imath x(Ax\wedge\neg Ax))$

2. $\imath x(Ax\wedge\neg Ax) = \imath x(Ax\wedge\neg Ax)$

3. $\forall x(Ax\wedge\neg Ax\leftrightarrow x = \imath x(Ax\wedge\neg Ax))$ 1, 2

4. $A(\imath x(Ax\wedge\neg Ax))\wedge\neg A(\imath x(Ax\wedge\neg Ax))\leftrightarrow \imath x(Ax\wedge\neg Ax) = \imath x(Ax\wedge\neg Ax))$ 3

5. $A(\imath x(Ax\wedge\neg Ax))\wedge\neg A(\imath x(Ax\wedge\neg Ax))$ 4, 2
\end{footnotesize}

\1

Similarly in the case of set-abstraction operator (where $R$ is $\in$) we obtain
just unrestricted axiom of comprehension which immediately leads to Russell's paradox. Hence it is crucial to establish what is $R$ for the specific tfo to decide if Tennant's rules may be safely added to GCI or GPCI. 
Therefore, we do not attempt here to state T-theory as a general calculus GT. Instead we will consider in the next section the application of his theory to set-abstraction operator, since even in this context one may introduce restrictions which can prevent us against troubles.

\section{Application to Set-abstracts}

Several kinds of set theory with set-abstraction operator as primitive can be rather easily developed on the basis of S- or T-theory as formalised in the preceding section. In fact, both Scott \cite{scott} and Tennant \cite{Tennant1978} applied their theories to set-abstract operators but in the context of free logic the unrestricted axiom of comprehension does not lead to Russell's paradox. However we work here in the setting of CFOL so the rules responsible for its derivation must be somehow restricted. For these reasons we decided to examine the possible formalisations of Quine's NF (New Foundations) as developed in \cite{Rosser}, where the comprehension axiom is suitably restricted by means of the outer syntactic side condition which is independent of the structure of rules. In fact, NF is not very popular formalisation of set theory due to some peculiarities. However, it has also several advantages which we are not going to discuss here because of the space restrictions\footnote{See in particular its presentation in \cite{Rosser} and discussion in \cite{Hatcher1982} and \cite{Holmes}.}. In particular,  the syntactic simplicity of NF make it a very convenient theory for proof-theoretic investigations.

Before we focus on sequent calculi for NF let us start with some general preliminaries concerning arbitrary formalisation of set theory. It often goes unnoticed that it may be developed in the language where only $\in$ is a primitive predicate or in the language with = primitive, which is rather more popular choice. In the latter case we assume that we have already  some axioms/rules for = , so the only specific axiom we need for sets is:

\1

$ExtAx:$ $\forall xy(\forall z(z\in x\leftrightarrow z\in y)\rightarrow x=y)$

\1

\n since the converse is already provable by LL.

\1

If we start with CFOL (only $\in$ primitive), = may be defined either in the Leibnizian spirit:

\1

$=^L$: \ $ t=t' := \forall z(t\in z\leftrightarrow t'\in z)$

\1

\n or in the way Quine prefers:

\1

$=^Q$: \ $ t=t' := \forall z(z\in t\leftrightarrow z\in t')$

\1

The first choice leads to the standard characterisation of = and the axiom $ExtAx$ is still required. The second one is different since $ExtAx$ is provable but still we cannot obtain the full characterisation of identity. Therefore  we must add a special form of LL as an extensionality axiom:

\1

$ExtAx'$: \ $\forall xyz(x=y \rightarrow (x\in z\rightarrow y\in z))$ 

\1

\n and this is the way Quine proceeded with the development of NF. The second axiom is the axiom of abstraction:

\1

$ABS$: \ $ \forall x(x\in \{y:\varphi(y)\} \leftrightarrow \varphi[y/x])$

\1

\n where $\varphi$ is stratified. Assuming that the only predicate is $\in$ this condition may be defined roughly as follows:  it is possible to define a mapping from variables of $\varphi$ into integers in a way that for each atom we have $i\in i+1$. In case we admit =, a mapping should yield $i=i$. In what follows we will admit both kinds of formulae as atomic, briefly called $\in$-atoms and =-atoms.

We will consider two approaches to construction of cut-free sequent calculus for NF. Although the rules $(Ext), (AV)$ will be not primitive but derivable in both, the first one, following closely Quinean formulation, is closer to the general GS, whereas the second starts with Tennant's rules suitably restricted.

\subsection{The S-approach to NF}

There is no sense to take the instances of $(Ext)$ and $(AV)$ as primitive rules since it will not save us from addition of most of the specific rules for set-abstraction operators and =. So it is better to follow quite closely the original Quinean axiomatisation of NF. A difference with the latter is connected with the treatment of identity, since we take it as a primitive predicate characterised by rules. However, we do not take the primitive rules of GPCI for identity as primitive but rather provide new rules based on $=^Q$. Hence we take GPC as the basis and add:

\begin{prooftree}
	\AxiomC{$a\in t, \Gamma \Rightarrow \Delta, a\in t'$}
	\AxiomC{$a\in t', \Gamma\Rightarrow \Delta, a\in t$}
	\LeftLabel{$(\Rightarrow =)$}
	\BinaryInfC{$\Gamma\Rightarrow\Delta, t = t'$}
\end{prooftree}

\begin{prooftree}
	\AxiomC{$\Gamma \Rightarrow \Delta, b\in t, b\in t'$}
	\AxiomC{$b\in t, b\in t', \Gamma\Rightarrow \Delta$}
	\LeftLabel{$(=\Rightarrow)$}
	\BinaryInfC{$t=t', \Gamma\Rightarrow\Delta$}
\end{prooftree}

\n These rules correspond to $=^Q$.
Moreover, we add two rules corresponding to the axiom ABS:

	\1

\begin{tabular}{ll}
	
	$(Abs\Rightarrow)$ \ $\dfrac{\varphi[x/t], \Gamma \!\!\Rightarrow \Delta
	}{t\in\{x:\varphi(x)\}, \Gamma \!\!\Rightarrow
		\Delta }$  & \ \ 
	$(\Rightarrow Abs)$ \ $\dfrac{\Gamma \!\!\Rightarrow
		\Delta, \varphi[x/t]}{\Gamma \!\!\Rightarrow \Delta, t\in\{x:\varphi(x)\}}$\\[16pt]

\end{tabular}

\n with $\varphi$ stratified. 

We omit easy proofs of the equivalence of stated rules with respective axioms: ABS and the object language counterpart of $=^Q$. Proofs of these axioms, as well as derivability of our rules in GPC enriched with axiomatic sequents expressing ABS and $=^Q$ are straightforward and similar to proofs from Theorem 1. Instead we will show that although we have neither $(Ext)$ nor $(AV)$ as primitive rules they are derivable in such a system for stratified $\varphi$.

\begin{lemma}Derivability of $(Ext)$ and $(AV)$
\end{lemma}
	
\begin{proof}:
	
\begin{scriptsize}
	
\begin{prooftree}
	\AxiomC{$\varphi(a), \Gamma \Rightarrow \Delta, \psi(a)$}
	\LeftLabel{$(Abs\Rightarrow Abs)$}
	\UnaryInfC{$a\in\{x:\varphi(x)\}, \Gamma\Rightarrow\Delta, a\in \{x:\psi(x)\}$}
	\AxiomC{$\psi(a), \Gamma\Rightarrow \Delta, \varphi(a)$}
	\UnaryInfC{$a\in\{x:\psi(x)\}, \Gamma\Rightarrow\Delta, a\in \{x:\varphi(x)\}$}
	\LeftLabel{$(\Rightarrow =)$}
	\BinaryInfC{$\Gamma\Rightarrow\Delta, \{x:\varphi(x)\}=\{x:\psi(x)\}$}
\end{prooftree}
\end{scriptsize}

\1

The proof of $(AV)$ or alternatively, of $(ExtAV)$ is similar. \qed
\end{proof}

\1  

But the rules $(\Rightarrow =)$ and $(=\Rightarrow)$
are not sufficient for obtaining the complete characterisation of identity in NF. In particular they are not sufficient for the case corresponding to the specific instance of LL expressed by the axiom $ExtAx'$ Note that in general we must be able to prove:

\begin{enumerate}
	\item $t=t', t''=t' \Rightarrow t=t''$
	\item $t=t', t''\in t \Rightarrow t''\in t'$
	\item $t=t', t\in t'' \Rightarrow t'\in t''$
	\end{enumerate}

With case 1 there is no problem; it is derivable by $(\Rightarrow =), (=\Rightarrow)$, similarly as other properties of =, including reflexivity and symmetry.
The case 2 would be provable by $(=\Rightarrow)$ provided instead of $b$ we are allowed to use any term $t''$. So this case is problematic and needs reformulation of the rules which in general destroys the subformula property and may be troublesome in proving the cut elimination theorem. The case 3 corresponds exactly to $ExtAx'$ and requires a separate rule which possibly covers also the case 2. 
To avoid troubles we might follow the general solution introduced for GS and use the rule $(2LL)$ as two-premiss right-sided rule but it does not work since $(Abs \Rightarrow)$ introduces an $\in$-atom as a principal formula in the antecedent. As a result while proving cut elimination we cannot make a reduction of the following cut instance:

\begin{prooftree}
	
	\AxiomC{$\Gamma\Rightarrow\Delta, t=t'$}
	\AxiomC{$\Gamma \Rightarrow \Delta, t'\in \{x:\varphi\}$}
	\LeftLabel{$(2LL)$}
	\BinaryInfC{$\Gamma \Rightarrow \Delta, t\in\{x:\varphi\}$}
	\AxiomC{$\varphi(t), \Pi \Rightarrow \Sigma$}
	\RightLabel{$(Abs\Rightarrow)$}
	\UnaryInfC{$t\in \{x:\varphi\}, \Pi \Rightarrow \Sigma$}
	\RightLabel{$(Cut)$}
	\BinaryInfC{$\Gamma, \Pi \Rightarrow \Delta, \Sigma$}
	
\end{prooftree}

It seems that in the presence of $(Abs\Rightarrow)$ and $(\Rightarrow Abs)$ the only solution is to add a  3-premiss version of LL:

\begin{prooftree}
	
	\AxiomC{$\Gamma\Rightarrow\Delta, t=t'$}
	\AxiomC{$\Gamma \Rightarrow \Delta, \varphi(t)$}
	\AxiomC{$\varphi(t'), \Gamma \Rightarrow \Delta$}
	\LeftLabel{$(3LL)$}
	\TrinaryInfC{$\Gamma \Rightarrow \Delta$}
	
\end{prooftree}

\n where $\varphi(t)$ and $\varphi(t')$ are either $t''\in t$ and $t''\in t'$ or $t\in t''$ and $t'\in t''$.  

Summing up we obtain a system GSNF which adds to GPC the following rules: $(=\Rightarrow), (\Rightarrow =), (Abs\Rightarrow), (\Rightarrow Abs)$ and $(3LL)$ ($(Ref)$ is derivable).

\begin{theorem}
	GSNF is an adequate formalisation of NF.
\end{theorem}

Moreover the cut elimination theorem can be proved for GSNF in a similar fashion as in \cite{Indrzejczak2020b} where similar solution was provided for sequent calculi for free description theories. Note however that the situation with the subformula property is even worse than in GS (GS') due to the presence of $(3LL)$. Is it possible to obtain a better formalisation of NF by means of Tennant's rules?

\subsection{The T-approach to NF}

If we want to apply the approach of Tennant to NF we have = as a primitive predicate not only present in the language but already characterised by specific rules so we start with GPCI and add the following 
Tennant's-style rules:

\begin{prooftree}
	\AxiomC{$\varphi(a), \Gamma \Rightarrow \Delta, a\in t$}
	\AxiomC{$a\in t, \Gamma \Rightarrow \Delta, \varphi(a)$}
	\LeftLabel{$(\Rightarrow :)$}
	\BinaryInfC{$\Gamma \Rightarrow \Delta, t=\{x:\varphi(x)\}$}
\end{prooftree}


\begin{prooftree}
	\AxiomC{$\Gamma \Rightarrow \Delta, \varphi(b)$}
	\AxiomC{$b\in t, \Gamma\Rightarrow \Delta$}
	\LeftLabel{$(:\Rightarrow)$}
	\BinaryInfC{$t=\{x:\varphi(x)\}, \Gamma\Rightarrow\Delta$}
\end{prooftree}

\begin{prooftree}
	\AxiomC{$\Gamma \Rightarrow \Delta, b\in t$}
	\AxiomC{$\varphi(b), \Gamma\Rightarrow \Delta$}
	\LeftLabel{$(:\Rightarrow)$}
	\BinaryInfC{$t=\{x:\varphi(x)\}, \Gamma\Rightarrow\Delta$}
\end{prooftree}

\n where $a$ is not in $\Gamma, \Delta, \varphi$, $t$ is any term and $\varphi$ is stratified.

Note that $(Ext)$ and $(AV)$ are derivable which follows from the proofs of EXT and AV presented in section 3.2. As we noticed there, also the axiom ABS is provable, so we do not need special rules $(Abs\Rightarrow), (\Rightarrow Abs)$ too.
We do not need to care even about the axiom $ExtAx$ since it is  provable:

\begin{tiny}
\begin{prooftree}
	
	\AxiomC{$c\in a\leftrightarrow c\in b, c\in a\Rightarrow\Delta, c\in b$}
	\LeftLabel{$(\forall\Rightarrow )$}
	\UnaryInfC{$\forall z(z\in a\leftrightarrow z\in b), c\in a \Rightarrow c\in b$}
	\AxiomC{$c\in a\leftrightarrow c\in b, c\in b\Rightarrow c\in a$}
	\UnaryInfC{$\forall z(z\in a\leftrightarrow z\in b), c\in b \Rightarrow c\in a$}
	\LeftLabel{$(\Rightarrow :)$}
	\BinaryInfC{$\forall z(z\in a\leftrightarrow z\in b) \Rightarrow a =\{x:x\in b\}$}
	\AxiomC{$c\in b \Rightarrow c\in b$}
	\RightLabel{$(\Rightarrow :)$}
	\UnaryInfC{$\Rightarrow b =\{x:x\in b\} $}
	\LeftLabel{$(2LL)$}
	\BinaryInfC{$\forall z(z\in a\leftrightarrow z\in b) \Rightarrow a = b$}
	\LeftLabel{$(\Rightarrow\rightarrow )$}
	\UnaryInfC{$\Rightarrow\forall z(z\in a\leftrightarrow z\in b)\rightarrow a=b$}
	\LeftLabel{$(\Rightarrow\forall)$}
	\UnaryInfC{$\Rightarrow\forall xy(\forall z(z\in x\leftrightarrow z\in y)\rightarrow x=y)$}
\end{prooftree}
\end{tiny}

It seems that T-approach is better than S-approach to NF since it is more economical. However, if we think about cut elimination we must consider carefully the problem of primitive rules for identity. Although we first stated that we add the special Tennant's-style rules to GPCI and we used $(2LL)$ in the above proof it seems that we cannot keep $(2LL)$ since in general we face the same problem with cut elimination as in the case of S-system illustrated in subsection 4.1. To prove the cut elimination theorem we must again either generally replace $(2LL)$ with $(3LL)$ or to follow the strategy introduced in \cite{Indrzejczak2021c} and separate the rules for LL dealing with special cases of atomic formulae. One possibility is to keep:

\begin{prooftree}
	
	\AxiomC{$\Gamma\Rightarrow \Delta, t=t'$}
	\AxiomC{$\Gamma\Rightarrow \Delta, \varphi(t)$}
	\LeftLabel{$(2LL')$}
	\BinaryInfC{$\Gamma\Rightarrow\Delta, \varphi(t')$}
\end{prooftree} 

\n for $\varphi$ being $\in$-atom and restrict $(3LL)$ only to =-atoms:

\begin{prooftree}
	
	\AxiomC{$\Gamma\Rightarrow \Delta, t=t'$}
	\AxiomC{$\Gamma\Rightarrow \Delta, t=t''$}
	\AxiomC{$t'=t'', \Gamma\Rightarrow\Delta$}
	\LeftLabel{$(3LL')$}
	\TrinaryInfC{$\Gamma\Rightarrow\Delta$}
\end{prooftree} 

This way we obtain a system GTNF which adds to GPC the rules: $(:\Rightarrow), (\Rightarrow :), (2LL'), (3LL'), (Ref)$. $(2LL')$ deals only with $\in$-atoms and all properties of identity are derivable by $(Ref)$ and $(3LL)$. 

\begin{theorem}
	GTNF is an adequate formalisation of NF.
\end{theorem}

The cut elimination theorem is provable for GTNF as well. Unfortunatelly, the situation with the subformula property is similar to that in the system GSNF from the preceding subsection.
However, there are possible some simplifications obtained by reduction of the applications of $(3LL')$ if 
at least two of $t, t', t''$ are parameters.
Consider the cases with at most one term $t$ complex:

\begin{enumerate}
\item $a=b, a=c \Rightarrow b=c$
\item $t=b, t=c \Rightarrow b=c$
\item $a=t, a=b \Rightarrow t=b$
\item $a=b, a=t \Rightarrow b=t$
\end{enumerate} 

$(2LL')$ may be modified to cover identities from case 1 and 2:

\begin{prooftree}

\AxiomC{$\Gamma\Rightarrow \Delta, t=t'$}
\AxiomC{$\Gamma\Rightarrow \Delta, \varphi(t)$}
\LeftLabel{$(2LL'')$}
\BinaryInfC{$\Gamma\Rightarrow\Delta, \varphi(t')$}
\end{prooftree} 

\n for $\varphi(t')$ being $\in$-atom or =-atom of the form $b=c$  (a third term in the premisses may be complex or a parameter).
For cases 3 and 4 we may add the rules:

\begin{prooftree}
	
	\AxiomC{$\Gamma\Rightarrow \Delta, a=t$}
	\AxiomC{$t=b, \Gamma\Rightarrow \Delta$}
	\LeftLabel{$(Tr)$}
	\BinaryInfC{$a=b, \Gamma\Rightarrow\Delta$}
\end{prooftree} 

or

\begin{prooftree}
	
	\AxiomC{$\Gamma\Rightarrow \Delta, a=t$}
	\AxiomC{$b=t, \Gamma\Rightarrow \Delta$}
\LeftLabel{$(E)$}
	\BinaryInfC{$a=b, \Gamma\Rightarrow\Delta$}
\end{prooftree}

Any of them will do the task. For example, if we take $(E)$ we have a direct proof of 4 and the following proof of 3:

\begin{prooftree}
	
	\AxiomC{$a=t\Rightarrow a=t$}
	\AxiomC{$b=t\Rightarrow b=t$}
	\AxiomC{$\Rightarrow b=b$}
	\AxiomC{$t=b\Rightarrow t=b$}
	\RightLabel{$(3LL')$}
	\TrinaryInfC{$b=t\Rightarrow t=b$}
	\RightLabel{$(E)$}
	\BinaryInfC{$a=t, a=b\Rightarrow t=b$}
\end{prooftree}

As a result we have to keep $(3LL')$ only for all cases where
at least two of $t, t', t''$ are complex terms at the price of adding $(Tr)$ or $(E)$.
Let us call such a modified system GTNF'.

\section{Conclusion}

We have provided a proof theoretic treatment of the general theory of tfos introduced independently by several authors (S-theory), and proposed a modification of a different approach (T-theory) in a way which allows us to compare their relative strength. Moreover, we examined the ways in which both approaches may be extended to cover set theory NF of Quine. All obtained sequent systems satisfy the cut elimination theorem,  but do not satisfy the subformula property. Hence, in the case of the systems for NF, we cannot obtain a syntactical consistency proof on the basis of the cut elimination theorem, because of the rules like $(3LL)$. Still these systems, in particular a system GTNF described in the last subsection, allow us to keep a stricter control over the construction of proof. 

The natural next step of this research is connected with the application of, possibly modified, systems GS, GS', or (suitably restricted) rules of Tennant's approach, to other kinds of term-forming operators, and careful examination of their specific features. 

Eventually it is also interesting to investigate if the obtained systems allow us to prove other desirable properties in constructive way. One of such important points is the interpolation theorem. Since it was proved semantically for the general S-theory in \cite{DaCosta80}, it is an important task to find a constructive proof as well.
However, the method of split-sequents due to Maehara, which is usually applied in the setting of sequent calculi, fails for the presented systems since it does not work with rules like $(a\Rightarrow)$. The problem is connected with the fact that the complex term occuring in the active formula in the premiss may contain some predicates which do not occur in the rest of the respective division of a split-sequent but occur in the interpolant (and of course in the other division of a split-sequent). In this case the interpolant of the premiss fails to be an interpolant of the conclusion, where the active formula is deleted. Only the weaker form of interpolation can be proved in which we require that interpolants have only  parameters (but not predicates) common to both divisions of the split-sequent. It is an open problem if such difficulties can be overcome.

\end{document}